\documentclass[11pt,reqno]{amsart}
\usepackage{amsmath}
\usepackage{amsthm}
\usepackage{amssymb}
\usepackage[all]{xy}
\usepackage{color,soul}
\usepackage{setspace}

\newcounter{obs}

\newtheorem{theorem}{Theorem}[section]
\newtheorem{proposition}[theorem]{Proposition}
\newtheorem{corollary}[theorem]{Corollary}
\newtheorem{lemma}[theorem]{Lemma}

\theoremstyle{definition}

\title{}

\title[Pietsch-Maurey-Rosenthal factorization]{Pietsch-Maurey-Rosenthal
factorization of summing multilinear operators}

\author[Mieczys{\l}aw~Masty{\l}o and Enrique \,A.~S\'{a}nchez P\'{e}rez]
{Mieczys{\l}aw~Masty{\l}o and Enrique \,A.~S\'{a}nchez P\'{e}rez}

\dedicatory{To the memory of Pawe{\l} Doma\'nski}

\address{M.~Masty{\l}o\\
Faculty of Mathematics and Computer Science\\
Adam Mickiewicz University in Pozna\'n\\
Umultowska 87, 61-614 Pozna\'n, Poland}

\email{mastylo@amu.edu.pl}

\address{E.\,A.~S\'{a}nchez P\'{e}rez\\
Instituto Universitario de Matem\'atica Pura y Aplicada,
Universitat Polit\`ecnica de Val\`encia\\ Camino de Vera s/n, 46022
Valencia, Spain.}

\email{easancpe@mat.upv.es}



\thanks{The first named author was supported by National Science Center, Poland, project
no. 2015/17/B/ST1/00064.  The second named author was supported by the Ministerio de Econom\'{\i}a
y Competitividad (Spain) under project MTM2016-77054-C2-1-P}

\date{}

\begin{document}

\maketitle

\begin{abstract}
The main purpose of this paper is the study of a~new class of
summing multilinear operators acting from the product of Banach
lattices with some nontrivial lattice convexity.~A~mixed
Pietsch-Maurey-Rosenthal type factorization theorem for these
operators is proved under weaker convexity requirements than the
ones that are needed in the Maurey-Rosenthal factorization through
products of $L^q$-spaces.~A~by-product of our factorization is an
extension of multilinear operators defined by a~$q$-concavity type
property to a~product of special Banach function lattices which
inherit some lattice-geometric properties of the domain spaces, as
order continuity and $p$-convexity. Factorization through
Fremlin's tensor products is also analyzed.~Applications are
presented to study a~special class of linear operators between
Banach function lattices that can be characterized by a strong
version of $q$-concavity. This class contains $q$-dominated
operators, and so the obtained results provide a~new factorization
theorem for operators from this class.
\end{abstract}

\vspace{7 mm}

\section{Introduction}

Domination inequalities for multilinear operators are of interest
in applications  to factorization of various types of operators
(see \cite{BoPeRuJMAA,defmas,defmaspoly,jca}). In the case of
operators defined on products of Banach lattices, these
dominations are  deeply related to Banach lattice geometric
notions, as $q$-convexity or $q$-concavity.  It should be pointed
out that domination does not lead in general to a~nice
factorization in the multilinear case. However, in some situations
the relation between domination and factorization works as in the
linear case. We recall that the famous Pietsch's
factorization theorem is given by a~domination result associated
to summability properties also in the multilinear case, in which
$L^p$-spaces are involved. We also point out that under the
assumption of some variants of convexity properties of the
involved lattices, the Maurey-Rosenthal multilinear theorem allows
to link a~$q$-concavity type domination inequality  with a~
factorization/extension of the multilinear operator.

In this paper we are concerned with the analysis of some new
lattice geometric properties that we call $p$-strong $q$-concavity
(see Section \ref{Sla2}). The motivation for this is to prove
domination/factorization characterizations for multilinear
operators from the product of Banach lattices that satisfy
a~certain vector norm inequality. Recall that in the case of
linear operators acting in Banach lattices, if an operator $T\colon X
\to Y$ is $q$-summing then it is also $q$-concave. This is the
main lattice-type property that is normally used when a
summability property for an operator among Banach lattices is
considered. Indeed, this implies ---using the Maurey-Rosenthal
factorization and under the assumption of $q$-convexity of the
domain lattice---, that the operator factors through an
$L^q$-space. Now take an index $1 \le p < q$ and write $r$ for the
real number satisfying that $1/q+1/r=1/p$. Then we can easily see
that
\[
\sup_{x^* \in B_{X^*}} \Big(\sum_{k=1}^{n} | \langle x_{k}, x^* \rangle |^q \Big)^{1/q} \le \sup_{x^* \in B_{X^*}}
\sup_{(\beta_k) \in B_{\ell^r}}  \Big(\sum_{k=1}^{n} | \beta_k \langle x_k, x^* \rangle |^p \Big)^{1/p}
\]
\[
\le
\sup_{(\beta_k) \in B_{\ell^r}} \Big\|  \Big(\sum_{k=1}^{n} | \beta_k  x_k  |^p \Big)^{1/p} \Big\|_X
\le
\Big\Vert\Big(\sum_{k=1}^n |x_k|^q\Big)^{1/q}\,\Big\Vert_X
\]
for every finite sequence $(x_k)_{k=1}^n$ in  the Banach lattice $X$.
A look to the definitions (see Section \ref{Sla2}) shows  that the implications
$$
\textit{p-summing } \Rightarrow   \textit{$p$-strongly $q$-concave}  \Rightarrow  \textit{$q$-concave}
$$
hold for operators acting in Banach lattices.

The main  advantage in using this new lattice property
---$p$-strong $q$-concavity--- is that \emph{the requirement on
the $q$-convexity of the original space can be relaxed} and still
obtain an standard factorization theorem. Indeed, Maurey-Rosenthal
theorem implies that $q$-summability of the operator \textit{plus}
$q$-convexity of the domain space allows a~strong factorization
through an $L^q$-space.  In the preliminary paper \cite{elprim},
it is shown that for  $1 \le p < q$, every $p$-strongly
$q$-concave operator ---and so every $q$-summing operator---
acting in a $p$-convex  space factors strongly through a Banach
function lattice space of the new class $ S_{X_p}^q(\xi)$, that
admit an easy description and whose lattice properties are
naturally associated to $p$-strongly $q$-concave operators. The
aim of this paper is to draw  the complete picture for this class
of lattice dominations/factorizations of operators by analyzing
their multilinear variants. By applying them to the linear case we
will show new factorization theorems for the classical
$q$-dominated (linear) operators among Banach lattices.

In Section 2 we sketch some background from the theory of general
Banach lattices, Fremlin's tensor product of Banach lattices, and
also summing operators. We also provide examples which motivates
our study.

In Section 3 we study a new class of summing multilinear operators
acting from the product of Banach lattices nontrivial lattice
convexity. We prove an extension theorem for these operators
acting in products of Banach lattices with some nontrivial
convexity. We give a mixed Pietsch-Maurey-Rosenthal type
factorization theorem for the multilinear case. We show that
a~particular class of multilinear operators defined by a
$q$-concavity type property can be extended to a~product of Banach
lattice satisfying some lattice-geometric properties, as order
continuity and $p$-convexity.

In section 4 we show the relation among summability of multilinear
operators from suitable products of Banach function lattices and
Fremlin tensor product. Factorization theorems are proved.

In Section 5 we center our attention on the factorization of the
linear dominated operators associated to a~new geometric
definition introduced in the paper. Indeed, the lattice-geometric
domination inequality appearing in the definition of the
$p$-strongly $q$-concave operators motivates the definition of the
dual notion.

\vspace{2 mm}

\section{Notation and background material} \label{Sla2}

The purpose of this section is to sketch some background from the
theory of general Banach lattices and summing operators. We shall
also take the opportunity to establish some
notation. For a given dual pair $\langle X, Y, (\cdot, \cdot)
\rangle$ the evaluation map $(x, y)$ is denoted by $\langle x,
y\rangle$ for all $x\in X$, $y\in Y$.

For notations concerning vector lattices we follow \cite{AB,LT},
and tensor products of Banach lattices we follow \cite{Fremlin1,Fremlin2}.
Let $(E, \leq)$ be a~vector lattice (called also a~Riesz space). If $A\subset E$,
then $A^{+}:=\{x\in A; \, x\geq 0\}$. Let us recall, that if $A$ is a~subset of a~Banach lattice
$E$, then a~functional $x^{*}\in E^{*}$ satisfying the condition $\langle x, x^{*} \rangle >0$
whenever $0<x \in A$ is called strictly positive on $A$. It is known that strictly positive
functionals on $E$ exist when $E$ has the order continuous norm and a~weak unit (see
\cite[Theorem 12.43]{AB} or \cite[Proposition 1.b.15]{LT}).

We also recall that a~Banach lattice possessing order continuous
norm and a~weak unit is order isomorphic to a~Banach function
lattice on a finite measure space.

We recall that if $(\Omega,\Sigma,\mu)$ is
a~\textit{$\sigma$-finite} measure space $L^0(\mu)$ denotes the
space of  $\mu$-a.e equal equivalence classes of functions.
A~\emph{Banach function lattice} is a~Banach space  $X\subset L^0(
\mu)$ with norm $\Vert\cdot\Vert_{X}$ such that if $f\in
L^0(\mu)$, $g \in X$ and $|f|\le|g|$ $\mu$-a.e.\ then $f\in X$ and
$\Vert f\Vert_X\le\Vert g\Vert_X$. Every Banach function space is
a Banach lattice with the pointwise $\mu$-a.e.\ order. The K\"othe
dual $X'$ of a Banach function space $X$ is the subspace of the
dual space $X^*$ of the functionals that has an integral
representation, that is, $x^{*} \in X^*$ for which there exists
$x' \in L^0(\mu)$ such that
\[
\langle x, x^{*}\rangle = \int_{\Omega} x x'\,d\mu, \quad\, x\in
X.
\]
In what follows we consider a dual pair $\langle X, X' \rangle$
with the evaluation map $(x, x') \mapsto \langle x, x'\rangle
:=\int_{\Omega} x x'\,d\mu$ for all $(x, x') \in X\times X'$.

A~Banach function lattice is said to have the \emph{Fatou
property} if for every sequence $(f_n)$ in $X$ such that $0\le
f_n\uparrow f$ a.e.\ and $\sup_n\Vert f_n\Vert_X<\infty$, it
follows that $f\in X$ and $\Vert f_n\Vert_X\uparrow\Vert
f\Vert_X$. This is equivalent to the fact that $X=X''$ with
equality of norms.

We use $\mathcal{M}(K)$ to denote the space of regular Borel
probability spaces on a~compact Hausdorff space. We recall that
the weak$^{*}$ topology on the dual $E^*$ a~Banach space $E$ is the topology
pointwise convergence. Then the unit ball is compact, by the
Banach-Alaoglu theorem.

The normed space $(\mathbb{R}^n, \|\cdot\|_p)$ is denoted by
$\ell_p^n$ for $1\leq p\leq \infty$, where as usual for any
$x=(t_1, \ldots, t_n) \in \mathbb{R}^n$,
\[
\|x\|_p = \Big(\sum_{k=1}^n |t_k|^p\Big)^{1/p},
\]
and
\[
\|x\|_{\infty} = \max_{1\leq k \leq n} |t_k|.
\]
In what follows the unit ball $B_{\ell_p^n}$ is denoted by
$B_{p}^n$ for short.

Given a~Banach lattice $X$, a~Banach space $Y$, and numbers $1\leq
p$, $q <\infty$. An operator $T\colon X \to Y$ is said to be
\emph{$q$-concave} if there exists  $C_{(q)} >0$ such that
\[
\Big(\sum_{k=1}^{n} \|Tx_k\|_{Y}^q \Big)^{1/q} \le C_{(q)}
\Big\Vert\Big(\sum_{k=1}^n |x_k|^q\Big)^{1/q}\,\Big\Vert_X
\]
for every choice of elements $x_1,..., x_n $ in $X$. The infimum of the
values $C_{(q)} $ for which the inequality above is satisfied will be
denoted by $M_{(q)} (T)$.

A~Banach lattice $X$ is said to be $p$-\emph{convex}, $1\leq p
<\infty$, respectively $q$-\emph{concave}, $1 \leq q < \infty$, if
there are positive constants $C^{(p)}$ and $C_{(q)}$ such that
\[
\Big\Vert\Big(\sum_{k=1}^n |x_k|^p\Big)^{1/p}\,\Big\Vert_X \le
C^{(p)} \Big(\sum_{k=1}^n\Vert x_k \Vert_X^p\Big)^{1/p},
\]
respectively,
\[
\Big(\sum_{k=1}^n \|x_k\|_{X}^q \Big)^{1/q} \le C_{(q)}
\Big\|\Big(\sum_{k=1}^n |x_k|^q\Big)^{1/q}\Big\|_{X}
\]
for every finite sequence $(x_k)_{k=1}^n$ in $X$. The least such
$C^{(p)}$ (respectively, $C_{(q)})$) is denoted by $M^{(p)}(X)$
(respectively, $M_{(q)}(X)$).~It is well-known that a $p$-convex
Banach ($q$-concave) lattice can always be renormed with a lattice
norm in such a way that $M^{(p)}(X)=1$ ($M_{(q)}(X)=1$).
\emph{Henceforth, throughout the paper we shall always assume that
$M^{(p)}(X)=1$}. We refer to \cite[Ch.\,1.d]{LT} or
\cite[Ch.2]{libro} for information about the classical geometric
concepts of (lattice) $p$-convexity and $q$-concavity.

If a~Banach lattice $X$ is $p$-convex with $1\leq p<\infty$, then
its $p$-\emph{concavification} is a~Banach lattice $X_p$ (see
\cite[p.~54]{LT} for details). Note that in the case of a Banach
function lattice $X$ on $(\Omega, \Sigma, \mu)$, $X_p$ is
identified with the space of all $f\in L^0(\mu)$ so that
$|f|^{1/p}\in X$ and equipped with the norm $\|f\|_{X_p} =
\||f|^{1/p}\|_{X}^p$.

We will use Fremlin tensor products of Banach lattices. Let
$X_1,\ldots, X_n$ and $Y$ be Archimedean Riesz spaces. An
$n$-linear map
\[
B\colon X_1\times \cdot\cdot\cdot \times X_n \to Y
\]
is called positive if $B(x_1,\ldots, x_n) \in Y^{+}$ whenever $x_k
\in X_k^{+}$, $1\leq k\leq n$; it is called a Riesz $n$-morphism
if $B(|x_1|, \ldots, |x_n|) = |B(x_1,\ldots, x_n)|$ for all
$x_k\in X_k$, $1\leq k\leq n$.

Following \cite{Fremlin2} (see also \cite{Schep}) one can
construct an Archimedean Riesz space $X_1
\bar{\otimes}\cdot\cdot\cdot \bar{\otimes}X_n$ and a~Riesz
morphism (called the \emph{Fremlin map}) $\bigotimes$.

We recall fundamental properties of this construction;

(a) $X_1 \otimes\cdot\cdot\cdot \otimes X_n$ is dense in $X_1
\bar{\otimes}\cdot\cdot\cdot \bar{\otimes} X_n$, i.e., for any $u
\in X_1 \bar{\otimes}\cdot\cdot\cdot \bar{\otimes} X_n$ there
exist $x_k\in X_{k}^{+}$ ($1\leq k\leq n$) such that for all
$\varepsilon > 0$ there is a~$v\in X_1
\bar{\otimes}\cdot\cdot\cdot \bar{\otimes} X_n$ with $|u- v| \leq
\varepsilon (x_1 \otimes\cdot\cdot\cdot \otimes x_k$).

(b) If $u \in  X_1 \bar{\otimes}\cdot\cdot\cdot \bar{\otimes}
X_n$, then there exist $x_k\in X_{k}^{+}$ ($1\leq k \leq n$) such
that $|u| \leq x_1 \otimes\cdot\cdot\cdot \otimes x_n$.

If $X_1, \ldots, X_n$ are Banach lattices, then we can define the
positive-projective norm $\|\cdot\|_{|\pi|}$ on $X_1
\bar{\otimes}\cdot\cdot\cdot \bar{\otimes} X_n$ by
\[
\|x\|_{|\pi|} = \inf \bigg\{\sum_{i=1}^n \sum_{j=1}^m \|x_{i,
j}\|_{X_j}; \, x_{i, j} \in X_{j}^{+}, \,\, |x| \leq \sum_{i=1}^n
x_{i, 1}\otimes\cdot\cdot\cdot \otimes x_{i, m}\bigg\}.
\]
We define the \emph{Fremlin tensor product} to be the Banach lattice
$X_1\otimes_{|\pi|}\cdot\cdot\cdot \otimes_{|\pi|} X_m$ given by the
completion of $X_1 \bar{\otimes}\cdot\cdot\cdot
\bar{\otimes} X_n$ with respect to $\|\cdot\|_{|\pi|}$.

We note that in the case of Banach function lattices $X_1,\ldots,
X_m$ on measure spaces $(\Omega_1, \Sigma_1, \mu_1),\ldots,
(\Omega_n, \Sigma_n, \mu_n)$, respectively, we can define the
Riesz space $X_1 \bar{\otimes}\ldots \bar{\otimes} X_n$ generated
by
\[
\{x_1\odot\cdot\cdot\cdot\odot x_{n}; \, x_j \in X_j, \, 1\leq
j\leq n\}
\]
in $L^0(\mu_1\times \cdot\cdot\cdot \times \mu_n)$, where
\[
(x_1\odot\cdot\cdot\cdot \odot x_{n})(\omega_1,\ldots, \omega_n)
:= x_1(\omega_1)\cdot\cdot\cdot x_n(\omega_n)
\]
for all $(x_1, \ldots, x_n) \in X_1 \times \cdots \times X_n$ and
$(\omega_1,\ldots, \omega_n) \in \Omega_1 \times \cdots\times
\Omega_n$.

Let us introduce now the notion that motivates the multilinear
definition given in this paper. Let $1\le p\le q<\infty$. Consider
$T\colon X\to E$ a~linear operator from a Banach lattice $X$ into
a~Banach space $E$. We will say that $T$ is \emph{$p$-strongly
$q$-concave} if there exists $C>0$ such that
\[
\Big(\sum_{k=1}^n\Vert Tx_{k}\Vert_E^q\Big)^{1/q}\le C
\sup_{(\beta_k) \in B_{r}^n}\Big\Vert\Big(\sum_{k=1}^n |\beta_k
x_k|^p\Big)^{1/p}\,\Big\Vert_X
\]
for every finite sequence $(x_k)_{k=1}^n$ in $X$, where $1<r\le\infty$
is such that $1/r = 1/p - 1/q$.

We present some examples showing the nature  of \textit{linear}
$p$-strongly $q$-concave operators. The reader can find more
examples in \cite{elprim}.

Fix $1\leq p <q <\infty$ and let $1/r = 1/p - 1/q$. Clearly that
$r/p$ and $q/p$ are conjugate exponents, that is, $1/(r/p) +
1/(q/p) = 1$. Since
\[
\sup_{(\beta_k) \in B_{r}^n} \Big\Vert\Big(\sum_{k=1}^n |\beta_k
x_k|^p\Big)^{1/p}\,\Big\Vert_X \le \Big\Vert\Big(\sum_{k=1}^n
|x_k|^q\Big)^{1/q}\,\Big\Vert_{X},
\]
it follows that a $p$-strongly $q$-concave operator is always $q$-concave.

Now observe that if $1\leq p\leq q<\infty$ and $X$ is a $p$-concave
Banach lattice, then the identity map $\iota\colon X\to X$ is
$p$-strongly $q$-concave. To see this we fix a~finite sequence
$(x_k)_{k=1}^n$ in a~$p$-concave
Banach function lattice $X$. Without loss of generality we may assume
that $M_{(p)}(X)=1$. Let $\alpha_k = \|x_k\|^{q/r}/( \sum_{k=1}^n \|x_k\|^q)^{1/r}$
for each $1\leq k \leq n$. Since $q= (pq/r) + p$, $\sum_{k=1}^n
\alpha_k^r =1$ and so
\begin{align*}
\Big(\sum_{k=1}^n \big\|x_k \big\|^q_X  \Big)^{1/q} & = \Big(
\sum_{k=1}^n \big\|x_k \big\|^{pq/r} \cdot \|x_k\|^p  \Big)^{1/q}
\le \Big\| \Big( \sum_{k=1}^n \big(|x_k| \|x_k\|^{q/r} \big)^p
\Big)^{1/p} \Big\|^{p/q} \\
& \le \sup_{(\beta_k) \in B_{r}^n} \Big\|\Big( \sum_{k=1}^n
|\beta_k \, x_k|^p \Big)^{1/p} \Big\|^{p/q} \Big( \sum_{k=1}^n
\|x_k\|^{q} \Big)^{p/(rq)}.
\end{align*}
Hence
\[
\Big(\sum_{k=1}^n \big\|x_k \big\|^q_X  \Big)^{p/q^2} = \Big(
\sum_{k=1}^n \big\|x_k \big\|^q_X  \Big)^{1/q-p/(rq)} \le
\sup_{(\beta_k) \in B_{r}^n} \Big\| \Big( \sum_{k=1}^n|\beta_k
\,x_k|^p \Big)^{1/p} \Big\|^{p/q},
\]
and this gives the above mentioned statement.

We note that the above observation shows that all $L^p$-spaces are
$p$-strongly $q$-concave, thus an operator acting in $L^p$-space
is so. However, there are of course other situations.

We show an example of a $p$-strongly $q$-concave operator acting
in a~$p$-convex Banach function lattice that is not an $L^p$
space. To see this we need to define special spaces and show some
preliminary results.

Let $1<p<q <\infty$ and let $p'$ be the conjugate number given by
$1/p'=1 -1/p$. Assume that $(\Omega,\Sigma,\mu)$ is
a~$\sigma$-finite measure space such that there is
a~measurable partition $(A_k)_{k=1}^\infty$ of $\Omega$ with
$\mu(A_k)=1$ for each $k$. Consider the sequence of characteristic
functions $(\chi_{A_k})_k$ and define an order continuous Banach
function lattice
\[
Y:= \Big\{f\in L^0(\mu); \, (f\xi_{A_k})_{k=1}^{\infty} \in
\big(\oplus L^p(\mu|_{A_k}  \big) \big)_{\ell_q} \Big\}
\]
equipped with the norm
\[
\|f\|_Y := \|(f\xi_{A_k})\|_{(\oplus L^{p}(\mu|_{A_k}))_{\ell_q}}=
\Big( \sum_{k=1}^\infty \Big( \int_{A_k} |f|^p \, d\mu \Big)^{q/p}
\Big)^{1/q}, \quad f \in Y.
\]
It is easy to check that
\[
\|\chi_{A_k}\|_{(Y_p)'} = 1, \quad\, k\in \mathbb{N}.
\]
This implies that for each $k\in \mathbb{N}$ we have a functional $x_{k}^{*}\in
B_{(Y_p)^{*}}$ given by
\[
x_{k}^{*}(f) = \int_{A_k} f\,d\mu, \quad\, f\in Y_{p}.
\]
Now observe that the linear map $R$ defined by
\[
R(f) = \Big( \frac{1}{2^{k/q}} \int_{A_k} f \, d\mu \Big)_k, \,
\quad\, f\in Y
\]
is bounded from $Y$ to $\ell_q$.

For the Borel regular measure on $B_{(Y_p)^{*}}$ given by $\nu =
\sum_{k=1}^\infty 2^{-k}\,\delta_{x_{k}^{*}}$, we denote by
$S^q_{Y_p}(\nu)$ the space of all $f\in L^0(\mu)$ such that
\[
\|f\|_{p,q; \nu} :=\Big(  \int_{B_{(Y_p)^{*}}^+} \big| \langle
|f|^p, y^{*} \rangle \big|^{q/p} \, d \nu(y^{*})\Big)^{1/q}
<\infty.
\]
A direct computation shows that
\[
\|f\|_{p, q, \nu} = \Big( \sum_{k=1}^\infty \frac{1}{2^k} \Big(
\int_{A_k} |f|^p \, d\mu \Big)^{q/p} \Big)^{1/q}, \quad f \in Y.
\]
Therefore, $\|\cdot\|_{p,q,\nu} \le \|\cdot\|_Y$ and the operator
$R$ can be extended to the space $S^{q}_{Y_p}(\nu)$, since
\[
\|R(f)\|_{\ell^q} = \Big\| \Big( \frac{1}{2^{k/q}} \int_{A_k} f \,
d\mu \Big)_k \Big\|_{\ell^q} \le \Big( \sum_{k=1}^\infty
\frac{1}{2^k} \Big( \int_{A_k} |f|^p \, d\mu \Big)^{q/p} \,
\mu(A_k)^{q/p'} \Big)^{1/q} = \|f\|_{Y}.
\]
Now observe that for any $f_1,\ldots, f_n \in Y$ we have
\begin{align*}
\sum_{k=1}^n \|T(f_k)\|_{\ell^q}^q  & \le \int_{B_{(Y_p)^{*}}^+}
\Big(\sum_{k=1}^n \big| \langle
|f_k|^p, y^{*} \rangle \big|^{q/p} \Big)\,d\nu(y^{*}) \\
& \le   \sup_{y^{*} \in {B_{(Y_p)^*}^+}} \sum_{k=1}^n
\big| \langle |f_k|^p, y^{*} \rangle \big|^{q/p} \\
& = \sup_{(\beta_k) \in B_{r}^n} \Big\Vert\Big(\sum_{k=1}^n
|\beta_k f_k|^p\Big)^{1/p}\,\Big\Vert_{Y}^q,
\end{align*}
where we have used Lemma \ref{S-space} (see Section \ref{S3} below).

\section{Summing multilinear operators on products of Banach lattices} \label{S3}

In what follows we assume that the $m$-tuples $(p_1,\ldots, p_m)$,
$(q_1, \ldots, q_m)$ and $(r_1,\ldots, r_m)$ of real numbers
satisfy $1\leq p_j \leq q_j$, $1/r_j = 1/p_j - 1/q_j$ for each
$1\leq j\leq m$. We also define $q$ by $1/q:= 1/q_1 + \ldots +
1/q_m$.

A~multilinear operator $T\colon X_1 \times \cdots \times X_m \to Y$ ---where $X_j$ is
a~Banach lattice and $Y$ is a~Banach space---,  is said to be $(p_1,
\ldots, p_m)$-strongly $(q_1,\ldots, q_m)$-concave whenever there
exists a~constant $C>0$ such that for any finite sequence
$(x^j_{k})_{k=1}^n$ in $X_j$, $1\leq j\leq m$, we have that
\[
\Big( \sum_{k=1}^n \| T(x^1_k, \ldots, x^m_k) \|_{Y}^q \Big)^{1/q}
\le C \prod_{j=1}^{m} \sup_{(\beta^{j}_k) \in B_{\ell^{r_j}}}
\Big\Vert\Big(\sum_{k=1}^n |\beta^j_k x_k^j |^{p_j}
\Big)^{1/{p_j}}\,\Big\Vert_{X_j}.
\]

\vspace{1.5 mm}

We will use a~lemma which is a~general version of Lemma 2 in
\cite{elprim}.

\begin{lemma}\label{lemaequi1}
Let $1\leq p<\infty$ and let $E$ be a~$p$-convex Banach lattice.
Then
\[
\sup_{(\beta_k) \in B_{r}^n}\Big\Vert\Big(\sum_{k=1}^n |\beta_k
x_k|^p \Big)^{1/p}\,\Big\Vert_{E}= \sup_{x^{*} \in
B_{{(E_p)}^*}^+} \Big(\sum_{k=1}^n \langle |x_k|^{p}, x^{*}
\rangle^{q/p}\,\Big)^{1/q}
\]
for every choice of $(x_k)_{k=1}^n$ in $E$ where $1/r = 1/p -
1/q$.
\end{lemma}

\begin{proof}
Fix a~finite set $\{x_1,\ldots, x_n\}$ of  $E$, and note that
$(\ell^{q/p})^{*} = \ell^{r/p}$, by $r/p + q/p = 1$. Then we have
that
\begin{align*}
\sup_{(\beta_k)\in B_{r}^n} \Big\Vert\Big(\sum_{k=1}^n |\beta_k
x_k|^p\Big)^{1/p}\, \Big\Vert_{E}^p & =
\sup_{(\beta_k)\in B_{r}^n}\Big\Vert\sum_{k=1}^n |\beta_k x_k|^p\,\Big\Vert_{{E}_p} \\
& = \sup_{(\beta_k)\in B_{r}^n}\sup_{x^{*}\in B_{{{(E}_p)}^*}^+}
\Big\langle \sum_{k=1}^n|\beta_k x_k|^p, \,
x^{*} \Big\rangle \\
& =  \sup_{x^{*} \in B_{{({E}_p)}^*}^+}\,\sup_{(\alpha_k)\in
B_{r/p}^{n}} \sum_{k=1}^n |\alpha_k| \big\langle |x_k|^p, \varphi \big\rangle \\
& = \sup_{x^{*} \in B_{{(E_p)}^*}^+}
\Big(\sum_{k=1}^n\Big(\big\langle |x_k|^p, \, x^{*} \big\rangle
\Big)^{q/p}\,\Big)^{p/q}.
\end{align*}
\end{proof}

Now we state our first main theorem.

\begin{theorem}
\label{domi1} Let $X_j$ be $p_j$-convex Banach lattices and let
$1\leq q_j <\infty$ for each $1\leq j\leq m$. If $1/q = 1/q_1 +
\ldots + 1/q_m$, then the following are equivalent statements about
a~multilinear operator $T$ from $X_1 \times \cdots \times X_m$ to
a~Banach space $E$.

\begin{itemize}
\item[(i)] $T$ is $(p_1, \ldots, p_m)$-strongly $(q_1,\ldots,
q_m)$-concave.

\item[(ii)] There is a~constant $C>0$ such that for every $(x_1,
\ldots, x_m)\in X_1 \times\cdot\cdot\cdot \times X_m$,
\[
\|T(x_1, \ldots, x_m)\|_Y \le C \prod_{j=1}^m
\Big(\int_{B_{{{((X_j)}_{p_j})}^*}^+} \big\langle |x_j|^{p_j}, \,
x^{*}_j \big\rangle^{{q_j}/{p_j}}\, d \nu_j(x^{*}_j)
\Big)^{1/{q_j}},
\]
where $\nu_j$ is a~probability Borel measure on the weak$^{*}$
compact set $B_{{((X_j)_p)}^*}^+$ for each $1\leq j\leq m$.
\end{itemize}
\end{theorem}

\begin{proof}
(i) $\Rightarrow$ (ii). Fix  finite sequences $(x^j_i)_{i=1}^n$ in
$X_j$ for each $1\leq j\leq m$.

First consider Lemma \ref{lemaequi1} (with $m=1$ and $E=X_j$ for
each $j$) for all the factors in the product of the left hand side
of the inequality that provides the definition of
$(p_1,...,p_m)$-strongly $(q_1,...,q_m)$-concave $m$-linear
operator. We obtain that the following inequality is equivalent to
the one in this definition
\[
\Big(\sum_{i=1}^n \|T(x^1_i,\ldots, x^m_i)\|_{Y}^q \Big)^{1/q} \le
C \prod_{j=1}^{m} \sup_{x^{*}_j \in B_{{((X_1)_{p_j})}^*}^+}
\Big(\sum_{i=1}^n\Big(\big \langle |x_i^{j}|^{p_j},\,x_j^{*}
\big\rangle \Big)^{q_j/p_j}\,\Big)^{1/q_j}.
\]

From this on, the proof uses some methods from \cite{defantposit,jca}. We only sketch the main
arguments for the convenience of the reader; using Young's inequality, we obtain that the inequality
above implies
\[
\sum_{i=1}^n \| T(x^1_i, \ldots, x^m_i) \|_{Y}^q \le C^q \,\,
\sum_{j=1}^{m}\bigg(\frac{q}{q_j}\,\sup_{x_j^{*} \in
B_{{((X_j)_{p_j})}^*}^+} \sum_{i=1}^n \Big( \big\langle |
x_i^j|^{p_j},\,x_j^{*} \big\rangle \Big)^{q_j/p_j}\bigg).
\]

We now define a~convex set of continuous real functions
\[
\psi\colon \mathcal{M}(B_{{((X_1)_{p_1})}^*}^+) \times
\cdot\cdot\cdot\times \mathcal{M}(B_{{((X_m)_{p_m})}^*}^+) \to
\mathbb R,
\]
each one associated to each finite sets of finite sequences as the
ones at the beginning of the proof, and given by the formula

\begin{align*}
\psi( \eta_1, \ldots, \eta_m) & :=  \sum_{i=1}^n \|T(x^1_i,\ldots, x^m_i)\|_{Y}^q  \\
& - C^q \,\sum_{j=1}^m \bigg(\frac{q}{q_j}\,
\int_{B_{{((X_j)_{p_j})}^*}^+} \sum_{i=1}^n \Big(\big\langle |
x_i^j|^{p_j}, \, x^{*}_j \big\rangle \Big)^{q_j/p_j} \,d
\eta_j(x^{*}_j) \bigg).
\end{align*}

Note that $\mathcal{M}(B_{{((X_j)_{p_j})}^*}^+)$ is a~compact set
with the product  topology defined by means of the weak$^{*}$
topology of the dual of each Banach space ${(X_j)}_{p_j}$ for each
$1\leq j\leq m$. Recall that these spaces are Banach as a
consequence of the requirement that each of them is $p_i$-convex.
The functions are continuous with respect to the product topology
and satisfy all the properties needed for applying Ky Fan's Lemma
(see, e.g., \cite[Lemma 9.10]{DJT}). This gives an element
\[
(\nu_1,\ldots, \nu_m) \in \mathcal{M}(B_{{((X_1)_{p_1})}^*}^+)
\times \cdots \times \mathcal{M}(B_{{((X_m)_{p_m})}^*}^+)
\]
satisfying
\[
\sum_{i=1}^n \| T(x^1_i, \ldots, x^m_i)\|_{Y}^{q} -  C^q \,\,
\sum_{j=1}^{m}\bigg(\frac{q}{q_j}\, \int_{B_{{((X_j)_{p_j})}^*}^+}
\sum_{i=1}^n   \Big( \big\langle |x_i^j|^{p_j}, \, x_j^{*}
\big\rangle \Big) \bigg)^{q_j/p_j} \,d\nu_j(x_j^{*}) \le 0.
\]

Now, the multilinearity of $T$ allows to use a~direct argument for
choosing the right constants for getting from this sum domination
the product domination that is written in (ii) (see \cite[Theorem
1]{defantposit} for the details).

To conclude it is enough to observe that the converse inequality
is obvious by using Lemma \ref{lemaequi1}.
\end{proof}

The next result deals with factorization of $(p_1,\ldots,
p_m)$-strongly $(q_1,\ldots, q_m)$-concave multilinear operators.
Motivated by the above result we define Banach lattices which are
connected with the obtained characterization of these operators.

Let $1\leq p\leq q<\infty$. For a~given $p$-convex Banach function
lattice $X$ and a~regular Borel probability measure $\nu$ on
$B_{(X_p)^{*}}$ equipped with the weak$^{*}$-topology we define on
$X$ a~functional by
\[
\|x\|_{p,q,\nu} := \Big(\int_{B_{(X_p)^{*}}^{+}} \langle |x|^p,
x^{*} \rangle^{q/p} d\,\nu(x^{*})\Big)^{1/q}, \quad\, x\in X.
\]
We put $F_{X_p}^{q}(\nu):= (X,\|\cdot\|_{p,q,\nu})$. Clearly that
$\rho(\cdot):=\|\cdot\|_{p,q,\nu}$ defines a~lattice seminorm on
$X$. If $N$ is the null ideal of $\rho$, i.e., $N= \{x\in X; \,
\rho(x)= 0\}$, then $X/N$ is a~normed lattice (under the natural
order) equipped with the norm
\[
\|[x]\| := \rho(x), \quad\, [x] \in X/N.
\]
The norm completion $\widetilde{F}_{X_p}^{q}(\nu)$ of $X/N$ with
respect to the above lattice norm is a~Banach lattice. Note that
$\|x\|_{p,q, \nu} \leq \|x\|_X$ for all $x\in X$ implies that the
map $i_X$ defined by
\[
\iota_X(x) = [x], \quad\, x\in X
\]
is bounded from $X$ to $\widetilde{F}_{X_p}^q(\nu)$.

\vspace{2 mm}

We have the following useful lemma.

\begin{lemma}
\label{S-space}  Let $1\leq p\leq q<\infty$ and let $X$ be
a~$p$-convex Banach lattice.
\begin{itemize}
\item[{\rm(i)}] If there exists a~strictly positive functional on
$X_p$, then for every $\nu \in \mathcal{M}(B_{(X_p)^*}^{+})$ there
exists $\widetilde{\nu} \in \mathcal{M}(B_{(X_p)^*}^{+})$ such
that $F_{X_p}(\widetilde{\nu}) = (X, \|\cdot\|_{p, q,
\widetilde{\nu}} )$ is a~normed lattice such that
\[
\frac{1}{2} \|x\|_{p,q, \nu} \leq \|x\|_{p, q, \widetilde{\nu}} \leq
\|x\|_X, \quad\, x\in X.
\]
\item[{\rm(ii)}] If $X$ is an order continuous Banach function lattice on
$(\Omega, \Sigma, \mu)$, then for every $\nu \in \mathcal{M}(B_{(X_p)^*}^{+})$ there exists
a~probability Borel measure $\xi \in \mathcal{M}(B_{(X_p)^{'}}^{+})$ such that the
completion of $F_{X_p}(\widetilde{\nu})$ ---for $\widetilde \nu$ as in {\rm(ii)}--- is an
order continuous Banach function lattice $(S_{X_p}^q(\xi), \|\cdot\|)$ on $(\Omega,
\Sigma, \mu)$ given by
\[
S_{X_p}^q(\xi):= \bigg\{f\in L^0(\mu); \, \, \|f\| =\bigg(\int_{B_{(X_p)'}^{+}} \bigg(\int_{\Omega}
|f(\omega)|^p\,h(\omega)\,d\mu(\omega)\bigg)^{q/p}\,d\xi(h)\bigg)^{1/q} < \infty \bigg\}
\]
and satisfying
\[
\frac{1}{2} \|x\|_{p,q, \nu} \leq \|x\| \leq \|x\|_X, \quad\, x\in X.
\]
\end{itemize}

\end{lemma}

\begin{proof}
(i). Let $x^* \in (X_p)^{*}$ be a norm one strictly positive functional on $X_p$ and let $\delta_{x^*}\in
\mathcal{M}(B_{(X_p)^*}^{+})$ be the associated Dirac measure. For a~given $\nu \in \mathcal{M}(B_{(X_p)^*}^{+})$,
we define $\widetilde{\nu} := 1/2(\nu + \delta_{x^*}) \in \mathcal{M}(B_{(X_p)^*}^{+})$. It is
obvious that $\widetilde{\nu}$ satisfies the required properties.

(ii). Our hypothesis that $X$ (and so $X_p$) is order continuous implies that
for every $x^{*}\in (X_{p})^*$ there exists unique $h=h_{x^*}\in
(X_p)'$ such that
\[
\langle |x|^p, x^{*} \rangle= \int _{\Omega} |x|^p h \,d\mu,
\quad\, x\in X
\]
with $\|x^{*}\|_{(X_p)^*} = \|h\|_{(X_p)'}$, and moreover the map
$(X_p)^{*} \ni x^{*} \mapsto h_{x^*}$ is an order isometrical
isomorphism. We denote the restriction of this map to
$B_{(X_p)^{*}}$ by $\varphi$. Clearly, $\varphi$ is
a~topological homeomorphism of $B_{(X_p)^*}$ equipped with the
weak$^{*} $ topology onto $B_{(X_p)'}$ equipped with the pointwise
topology induced by $\sigma(X', X)$.

For a~fixed $\nu \in \mathcal{M}(B_{(X_p)^*}^+)$, we define
$\nu_{\varphi} \in \mathcal{M}(B_{(X_p)'}^+)$ by
$\nu_{\varphi}(A):= \nu(\varphi^{-1}(A))$ for any Borel subset of
$B_{(X_p)'}^+$. Then for every $x\in X$, we get that
\begin{align*}
\int_{B_{(X_p)^*}^+} \langle |x|^p, x^{*}\rangle^{q/p}\, d\nu(x^{*})
& = \int_{\varphi^{-1}(B_{(X_p)'}^+)} \langle |x|^p,
\varphi^{-1}(h_{x^{*}}) \rangle^{q/p} \,d\nu(x^{*}) \\
& =\int_{B_{(X_p)'}^+} \bigg(\int_{\Omega} |x|^p
h d \mu \bigg)^{q/p} \,d\nu_{\varphi}(h).
\end{align*}
Since $(X_p)'$ is a~Banach function lattice on $(\Omega, \Sigma,
\mu)$ there exists $h\in B_{(X_p)'}$ with $h>0$ on $\Omega$. Then
$\xi:= 1/2(\nu_{\varphi} + \delta_h ) \in \mathcal{M}(B_{(X_p)'})$, where
$\delta_{h}$ is a~Dirac measure generated by $h$.

Combining the above formula with \cite[Proposition 1]{elprim}, we
conclude that
\[
(S_{X_p}^q(\xi), \|\cdot\|)
\]
is the desired Banach
function lattice.
\end{proof}

\vspace{2 mm}

We are now ready to state the following factorization theorem.

\vspace{2 mm}

\begin{theorem}
\label{prifa}  Let $X_j$ be $p_j$-convex Banach function lattices
and let $1\leq q_j <\infty$ for each $1\leq j\leq m$. If $1/q =
1/q_1 + \ldots + 1/q_m$, then the following are equivalent
statements about a~multilinear operator $T$ from $X_1 \times
\cdots \times X_m$ to a~Banach space $Y$.

\begin{itemize}
\item[(i)] $T$ is $(p_1, \ldots, p_m)$-strongly $(q_1, \ldots,
q_m)$-concave.

\item[(ii)] There are probability Borel measures $\nu_j$ in
$\mathcal{M} (B_{((X_j)_{p_j})^*} )$ for each $1\leq j\leq m$, and
a~multilinear operator $S$ such that $T$ factors as
\[
\xymatrix{{X_1 \times \cdots \times X_m}
\ar[r]^{\,\,\,\,\,\,\,\,\,\,\,\,\,\,\,\,\,\,\,\,\,\,
T} \ar@{->}[d]_{\iota_1 \times \cdots \times \iota_m }  & Y \\
\widetilde{F}^{q_1}_{X_{p_1}}(\nu_1) \times \cdots \times
\widetilde{F}^{q_m}_{X_{p_m}}(\nu_m) \ar[ur]_{S} &  & }
\]
where $\iota_j = \iota_{X_j}$ for each $1\leq j \leq m$.

\end{itemize}
\end{theorem}

\begin{proof}
(i) $\Rightarrow$ (ii). From Theorem \ref{domi1}, it follows that
there is a~constant $C>0$ such that for every $(x_1, \ldots, x_m)
\in X_1 \times \cdot\cdot\cdot \times X_m$,
\[
\|T(x^1, \ldots, x^m)\|_Y \le C \prod_{j=1}^m
\Big(\int_{B_{{{((X_j)}_{p_j})}^*}^+} \big\langle |x_j|^{p_j}, \,
x^{*}_j \big\rangle^{{q_j}/{p_j}}\, d \nu_j(x^{*}_j)
\Big)^{1/{q_j}},
\]
where $\nu_j$ is a~probability Borel measure on the weak$^{*}$-
compact set $B_{{((X_j)_p)}^{*}}^{+}$ for each $1\leq j\leq m$. This
implies that
\[
\|T(x_1, \ldots, x_m)\|_Y \le C \rho_1(x_1) \cdot \cdot\cdot
\rho_m(x_m)
\]
holds for all $(x_1,\ldots, x_m) \in X_1 \times \cdot\cdot\cdot
\times X_m$ with $\rho_j(\cdot) = \|\cdot\|_{P_j, q_j, \nu_j}$ for
each $1\leq j\leq m$. In particular, this implies that the formula
\[
T_0([x_1], \ldots, [x_m]) := T(x_1, \ldots, x_m), \quad\,
(x_1,\ldots, x_m) \in X_1 \times\cdot\cdot\cdot\times X_m
\]
defines a~bounded multilinear operator from $X_1/N_1 \times
\cdot\cdot\cdot \times X_m/N_m$ to $Y$, where $N_j = \{x \in
X_j;\, \rho_j(x) = 0\}$ for each $1\leq j\leq m$. Denote by $S$
the unique multilinear continuous extension of $T_0$ to
$\widetilde{F}_{X_{p_1}}^{q_1}\times \cdot\cdot\cdot \times
\widetilde{F}_{X_{p_m}}^{q_m}$. Clearly we have  that $\iota_1 \times
\cdot\cdot\cdot \times \iota_m$ given by
\[
(\iota_1 \times\cdot\cdot\cdot\times \iota_m)(x_1,\ldots, x_m): = ([x_1],
\ldots, [x_m])
\]
for all $(x_1, \ldots, x_m) \in X_1\times\cdot\cdot\cdot \times
X_m$ is a bounded linear operator from
$X_1\times\cdot\cdot\cdot\times X_m$ to
$\widetilde{F}_{X_{p_1}}^{q_1}(\nu_1) \times \cdot\cdot\cdot
\times \widetilde{F}_{X_{p_m}}^{q_m}(\nu_m)$ and so we have the
required factorization
$$
T = S\circ (\iota_1 \times\cdot\cdot\cdot
\times \iota_m).
$$

The implication (ii) $\Rightarrow$ (i) is obvious.
\end{proof}

Combing the above corollary with Lemma \ref{S-space} we obtain the
following result for the case of order continuous Banach function
lattices.

\begin{corollary}
Let $X_j$ be order continuous $p_j$-convex Banach function lattices on measure
spaces $(\Omega_j, \Sigma_j, \mu_j)$ and let $1\leq q_j <\infty$,
$1\leq j\leq m$. If $1/q = 1/q_1 + \ldots + 1/q_m$, then the
following are equivalent statements about an $m$-linear operator
$T$ from $X_1 \times \cdots \times X_m$ to a~Banach space $Y$.
\begin{itemize}
\item[(i)] $T$ is $(p_1, \ldots, p_m)$-strongly $(q_1, \ldots,
q_m)$-concave. \item[(ii)] There are probability measures $\nu_j$
in $\mathcal{M} (B_{((X_j)_{p_j})^{'}} )$ for each $1\leq j\leq
m$, and a~multilinear operator $S$ such that $T$ factors
through the product of Banach function lattices
$S^{q_j}_{X_{p_j}}(\nu_j)$ on the corresponding measure spaces
$(\Omega_j, \Sigma_j, \mu_j)$ as
\[
\xymatrix{{X_1 \times \cdots \times X_m}
\ar[r]^{\,\,\,\,\,\,\,\,\,\,\,\,\,\,\,\,\,\,
T} \ar@{->}[d]_{\iota_1 \times \cdots \times \iota_m }  & Y \\
S^{q_1}_{X_{p_1}}(\nu_1) \times \cdots \times
S^{q_m}_{X_{p_m}}(\nu_m) \ar[ur]_{S} &  & }
\]
where $\iota_j\colon X_j \to S^{q_j}_{X_{p_j}}(\nu_j)$ are
continuous inclusions for  $1\leq j\leq m$.
\end{itemize}
\end{corollary}

\section{Domination and the Fremlin tensor product}

In this section we show  the relation among summability of
multilinear operators from suitable products of Banach function
lattices and Fremlin tensor products. This will provide  a  class
of multilinear operators which is different of the one analyzed in
the previous section. The main difference is that the
factorization is in the present case defined by a multilinear
operator with values in a tensor product structure and a linear
map, in an opposite way as what happens with the class of $(p_1,
\ldots, p_m)$-strongly $(q_1,\ldots, q_m)$-concave operators.

\vspace{2 mm}

\begin{theorem}
\label{th2}
Let $T\colon X_1 \times \cdots \times X_m \to Y$ be a~Banach space valued multilinear
operator, where $X_j$, $1\leq j\leq m$, are Banach lattices. Suppose that
$X_1\otimes_{|\pi|} \cdots \otimes_{|\pi|} X_m$ is embedded in the $p$-convex Banach
lattice $E$. Then the following statements are equivalent.
\begin{itemize}
\item[(i)] There is a~constant $C>0$ such that for each $1\leq
j\leq m$ and every choice of $(x^j_i)_{i=1}^n$ in $X_j$,
\[
\Big(\sum_{i=1}^n \| T(x^1_i, \ldots, x^m_i) \|_Y^q \Big)^{1/q}
\le C \sup_{(\beta_i) \in B_{r}^n}\Big\Vert\Big(\sum_{i=1}^n
\big|\beta_i \big(x_i^1 \otimes \cdot\cdot\cdot\otimes x_i^m\big)\big|^p
\Big)^{1/p}\,\Big\Vert_{E}.
\]

\item[(ii)] There is a constant $C>0$ such that for every
$(x_1,\ldots, x_m) \in X_1 \times \cdots \times X_m$,
\[
\|T(x_1,\ldots, x_m)\|_Y \le C \Big( \int_{B_{(E_p)^{*}}^+}
\big\langle | x_1 \otimes\cdot\cdot\cdot\otimes x_m |^p, \,
x^{*}\big\rangle^{q/p}\, d \nu \Big)^{1/q},
\]
where $\nu$ is a probability Borel measure on the weak* compact
set $B_{{({E}_p)}^*}^+$.
\end{itemize}
\end{theorem}

\begin{proof}
The argument follows the lines of the one given for Theorem
\ref{domi1}. The $p$-convexity of $E$ implies that $E_p$ is
a~Banach lattice. By the inclusion of the Fremlin tensor product
$X_1 \otimes_{|\pi|} \cdot\cdot \cdot \otimes_{|\pi|} X_m
\hookrightarrow E$ we have that all the tensors $x_1\otimes
\cdot\cdot\cdot \otimes x_m$ are in $E$, and so $|x_1\otimes
\cdot\cdot\cdot \otimes x_m|^p$ define a~continuous function in
$C(B_{(E_p)^{*}})$, where $B_{(E_p)^*}$ is equipped with the
induced topology by the weak$^*$ topology of $(E_p)^{*}$. From
this point on, the proof using Ky Fan's Lemma is similar to the
one of Theorem \ref{domi1}, using Lemma \ref{lemaequi1} for
defining the right set of functions $\phi\colon
\mathcal{M}(B_{(E_p)^*}) \to \mathbb{R}$, where only functions as
$|x_1\otimes \cdot\cdot\cdot \otimes x_m|^p$ are considered.

The converse implication is easily obtained by a~direct
calculation.
\end{proof}

\vspace{1.5 mm}

In the case of Banach function lattices on measure spaces we
obtain the following results on factorization of multilinear
operators.

\begin{corollary}
\label{co2} Let $T\colon X_1 \times
\cdots \times X_m \to Y$ be a~Banach space valued multilinear
 operator, where $X_j$ are Banach function lattices on
$(\Omega_j, \Sigma_j, \mu_j)$, $1\leq j\leq m$. Suppose
 that $X_1\otimes_{|\pi|} \cdots \otimes_{|\pi|} X_m$ is continuously embedded in
$E$, where $E=E(\mu_1 \times \cdots \times \mu_m)$
is a~$p$-convex Banach function lattice on the product measure
space. Then the following statements are equivalent.
\begin{itemize}
\item[(i)] There is a constant $C>0$ such that for each $1\leq
j\leq m$ and for every choice of sequences $(x^j_i)_{i=1}^n$ in $X_j$,
\[
\Big(\sum_{i=1}^n \| T(x^1_i, \ldots, x^m_i) \|_Y^q \Big)^{1/q}
\le C \sup_{(\beta_i) \in B_{r}^n}\Big\Vert\Big(\sum_{i=1}^n
\big|\beta_i\big(x_i^1 \odot \cdot\cdot\cdot\odot x_i^m)\big|^p
\Big)^{1/p}\,\Big\Vert_{E}.
\]

\item[(ii)] There is a constant $C>0$ such that for every
$(x_1,\ldots, x_m) \in X_1 \times \cdots \times X_m$,
\[
\|T(x_1,\ldots, x_m)\|_E \le C \Big(\int_{B_{(E_p)^{*}}^+}
\big\langle | x_1 \odot\cdot\cdot\cdot\odot x_m |^p, \, x^{*}
\big\rangle^{q/p}\, d \nu \Big)^{1/q},
\]
where $\nu$ is a probability Borel measure on the weak* compact
set $B_{{({E}_p)}^*}^+$.
\end{itemize}
\end{corollary}

Using the same proof but changing single tensors $x_1\otimes
\cdot\cdot\cdot \otimes x_m$ by finite combinations of these
products, we obtain the corresponding factorization theorem.

\begin{corollary}
Under the assumptions of Theorem \ref{th2} on the spaces $X_1,
\ldots, X_m$, $E$ and the multilinear operator $T\colon X_1\times
\cdot\cdot\cdot \times X_m \to Y$, the following statements are
equivalent.
\begin{itemize}
\item[(i)] There is a constant $C>0$ such that for each $1\leq
j\leq m$ and for every choice of matrices
$(x^j_{i,k})_{i=1,k=1}^{N,M}$ in $X_j$,
$(\lambda_{i,k})_{i=1,k=1}^{N,M}$ in $\mathbb{R}$,
\begin{align*}
\Big( \sum_{i=1}^N \Big\|& \sum_{k=1}^M \lambda_{i,k}
T(x^1_{i,k},\ldots, x^m_{i,k})\Big\|_{Y}^q \Big)^{1/q} \\
& \le C \sup_{(\beta_i) \in B_{r}^n}\Big\Vert\Big(\sum_{i=1}^N
\Big|\beta_i \Big(\sum_{k=1}^M \lambda_{i,k} \big(x_{i,k}^1\otimes
\cdot\cdot\cdot \otimes x_{i,k}^m\big) \Big)\Big|^p
\Big)^{1/p}\,\Big\Vert_{E}.
\end{align*}

\item[(ii)] The operator $T$ admits the following factorization
\[
\xymatrix{{X_1 \times \cdots \times X_m}
\ar[r]^{\,\,\,\,\,\,\,\,\,\,\,\,\,\,\,\,\,\,\,\,\,\,\,\,\,\,\,\,\,\,\,\,\,\,\,
T} \ar@{->}[d]_{\bigotimes}  & E \\
\widetilde{F}^{q}_{E_{p}}(\nu) \ar[ur]_{S} & & }
\]
where $\nu$ is a probability Borel measure on the weak* compact
set $B_{{({E}_p)}^*}^+$ and $\bigotimes$ is the Fremlin map.
\end{itemize}
\end{corollary}

\section{Factorization of $p$-strongly $q$-dominated operators.}

In this section we prove a factorization theorem for a special class
of linear operators between Banach lattices. We start with the following
definition. Let $1 \le q \le \infty$ and $1/q+1/q'=1$. Let
$1 \le p_1 \le q_1=q$, $1\le p_2 \le q_2 = q'$ and
$1/r_1 = 1/p_1 - 1/q_1$, $1/r_2 = 1/p_2 - 1/q_2$. An operator
$T\colon X \to Y$ between Banach lattices is said to be
\textit{$(p_1,p_2)$-strongly $(q_1,q_2)$-concave} whenever
\[
\Big|\sum_{k=1}^n \langle Tx_{k}, y_{k}^{*} \rangle \Big| \le C
\sup_{(\alpha_k) \in B_{r_1}^n} \Big\Vert\Big(\sum_{k=1}^n
|\alpha_k x_k|^{p_1} \Big)^{1/{p_1}}\,\Big\Vert_{X}\,
\sup_{(\beta_k) \in B_{r_2}^n} \Big\Vert\Big(\sum_{k=1}^n |\beta_k
y_k^*|^{p_2} \Big)^{1/{p_2}}\,\Big\Vert_{Y^{*}},
\]
for every choice of sequences  $(x_k)_{k=1}^n$ in $X$ and
$(y_{k}^{*})_{k=1}^n$ in $Y^{*}$.

We note that general examples of $(p_1,p_2)$-strongly
$(q_1,q_2)$-concave operators are given by the classical
$q$-dominated operators. Indeed, an operator $T\colon X \to Y$ is
said to be $q$-dominated ($1 \leq q <\infty$) if
\[
\Big|\sum_{k=1}^n \langle T(x_k),y_{k}^{*} \rangle\Big| \le C
\sup_{x^{*} \in B_{X^*}} \Big(\sum_{k=1}^n |\langle x_k, x^*
\rangle |^{q} \Big)^{1/{q}}\,\sup_{y^{**} \in B_{Y^{**}}}
\Big(\sum_{k=1}^n |\langle y_k^*, y^{**} \rangle |^{q'}
\Big)^{1/{q'}}
\]
for every choice of $(x_k)_{k=1}^n$ in $X$ and
$(y_{k}^{*})_{k=1}^n$ in $Y^{*}$.

Since $1/q= 1 -1/q'$, $1/q'= 1 -1/q$ and
\begin{align*}
\sup_{x^{*} \in B_{X^*}}&\Big(\sum_{k=1}^n |\langle x_k, x^*
\rangle|^{q} \Big)^{1/{q}}\, \cdot \, \sup_{y^{**} \in B_{Y^{**}}}
\Big(\sum_{k=1}^n |\langle y_k^*, y^{**} \rangle |^{q'}
\Big)^{1/{q'}} \\
& \leq \sup_{(\alpha_k) \in B_{q'}^n} \Big\Vert \sum_{i=1}^n
|\alpha_k x_k |\,\Big\Vert_{X}\, \cdot \,\sup_{(\beta_k) \in B_{q}^n}
\Big\Vert \sum_{k=1}^n |\beta_k y_k^{*}| \,\Big\Vert_{Y^*},
\end{align*}
we conclude that $T\colon X \to Y$ is $(1,1)$-strongly
$(q,q')$-concave operator.

\vspace{2 mm}

Before going on to results, let us observe the following fact. Suppose that
$T\colon X \to Y$ is an operator between Banach lattices such that $X$ is
$p_1$-convex and $Y$ is $p_{2}'$-concave with $1/p_2 + 1/p_2'=1$.
Then it follows from Theorem \ref{domi1} that $T$ is
$(p_1,p_2)$-strongly $(q_1,q_2)$-concave operator if and only if
there exist $C>0$ and probability measures $\nu_1 \in
\mathcal{M}(B_{(X_{p_1})^{*}})$ and $\nu_2 \in
\mathcal{M}(B_{((Y^*)_{p_2})^*})$ such that for every $(x, y^{*})
\in X \times Y^{*}$,
\begin{align*}
|\langle Tx, y^{*} \rangle| & \le C \bigg(
\int_{B_{{(X_{p_1})^{*}}}^+} \big\langle |x|^{p_1},\, x^{*}
\big\rangle^{{q_1}/{p_1}}\,d \nu_1(x^{*}) \bigg)^{1/{q_1}} \\
& \times \bigg( \int_{B_{((Y^{*})_{p_2})^*}^+} \big\langle
|y^{*}|^{p_2}, \, y^{**} \big\rangle^{{q_2}/{p_2}}\,d
\nu_2(y^{**}) \bigg)^{1/{q_2}}.
\end{align*}

We need the following lattice formula (see \cite[Proposition
12.6]{Schwarz}).

\vspace{2.5 mm}

\begin{proposition}
\label{latformula} Let $E$ be a~Banach lattice, $x_k \in E$
$(1\leq k\leq n)$, and $1\leq p\leq \infty$. Then
\[
\Big\|\Big(\sum_{k=1}^{n} |x_k|^p\Big)^{1/p}\Big\|_E = \sup
\bigg\{\sum_{k=1}^n \langle x_k, x_k^{*}\rangle ; \, x_k^{*}\in
E^{*}, \,\,\Big\|\Big(\sum_{k=1}^{n}
|x_k^{*}|^{p'}\Big)^{1/p'}\Big\|_{E^{*}} \leq 1\bigg\}.
\]
\end{proposition}

\vspace{2.5 mm}

An application of the above proposition is the following corollary.

\vspace{2.5 mm}

\begin{corollary} \label{otrafor}
Let $1 \le p_1 < q$, $ 1 \le p_2 < q'$, and let $r_1$ and $r_2$ be
given by $1/r_1= 1/p_1-1/q$ and $1/r_2= 1/p_2-1/q'$. Assume that
$T\colon X\to Y$ is an operator between Banach lattices such that
$X$ is $p_1$-convex and $Y$ be $p'_2$-concave. If there exists
a~constant $C>0$ such that for every sequence $(x_k)_{k=1}^n$,
\[
\inf_{(\alpha_k) \in B_{r_2}^n} \Big\| \Big(\sum_{k=1}^n \Big|
\frac{T(x_k)}{\alpha_k} \Big|^{p'_2} \Big)^{1/p'_2} \Big\|_Y \le C
\sup_{(\beta_k) \in B_{r_1}^n} \Big\| \Big(\sum_{k=1}^n |\beta_k
x_k|^{p_1} \Big)^{1/p_1} \Big\|_{X},
\]
then $T$ is $(p_1,p_2)$-strongly $(q,q')$-concave.
\end{corollary}

\vspace{2 mm}

\begin{proof}
From Proposition \ref{latformula}, it follows that it is
enough to show that

\begin{align*}
\sup & \bigg\{\Big|\sum_{k=1}^n \langle T(x_k),y_k^* \rangle\Big|;
\, \,y_k^{*} \in Y^{*}, \,\, \sup_{(\gamma_k) \in B_{r_2}^n}
\Big\Vert\Big(\sum_{i=1}^n |\gamma_k y_k^{*}
|^{p_2} \Big)^{1/{p_2}}\,\Big\Vert_{Y^*} \le 1 \bigg\} \\
& \le C \sup_{(\beta_k) \in B_{r_1}^n} \Big\Vert\Big(\sum_{k=1}^n
|\beta_k x_k|^{p_1} \Big)^{1/{p_1}}\,\Big\Vert_{X}
\end{align*}
for every choice of  finite sequences $(x_k)_{k=1}^n$ in $X$ and
$(y_{k}^{*})_{k=1}^n$ in $Y^{*}$.

Assume that $p_2 >1$; the proof for $p_2=1$ is the same with the
obvious changes in the computations. Fix now $(\alpha_k)$ in $B_{r_2}^n$.
We have
\begin{align*}
\sup & \bigg\{\Big|\sum_{k=1}^n \langle T(x_k),y_{k}^*
\rangle\Big| ; \, \,y_{k}^{*} \in Y^{*}, \, \, \sup_{(\gamma_k)
\in B_{r_2}^n} \Big\Vert\Big(\sum_{k=1}^n |\gamma_k y_k^{*}|^{p_2}
\Big)^{1/{p_2}}\,\Big\Vert_{Y^*} \le 1 \bigg\} \\
& \leq \sup \bigg\{\Big|\sum_{k=1}^n \langle T(x_k), y_k^{*}
\rangle\Big|;\, \,y_{k}^{*} \in Y^{*}, \, \, \,
\Big\Vert\Big(\sum_{k=1}^n |\alpha_k y_k^{*} |^{p_2}
\Big)^{1/{p_2}}\,\Big\Vert_{Y^*} \le 1 \bigg\}.
\end{align*}
Thus, we get that
\begin{align*}
\sup & \bigg\{\Big|\sum_{k=1}^n \langle T(x_k),y_k^* \rangle\Big|
;\, \, \sup_{(\gamma_k) \in B_{r_2}^n} \Big\Vert\Big(\sum_{k=1}^n
|\gamma_k y_k^{*} |^{p_2} \Big)^{1/{p_2}}\,\Big\Vert_{Y^*} \le 1
\bigg\} \\
& \leq \inf_{(\alpha_k) \in B_{r_2}^n} \Big\| \Big(\sum_{k=1}^n
\Big| \frac{T(x_k)}{\alpha_k} \Big|^{p'_2} \Big)^{1/p'_2}
\Big\|_{Y}.
\end{align*}
Combining with Proposition \ref{latformula}, this completes the proof.
\end{proof}

\vspace{2 mm}

Note that Corollary \ref{otrafor} allows to show that some classical
operators are $(p_1,p_2)$-strongly $(q_1,q_2)$-concave.
Consider the following example. Let $([0,1],\mathcal{B}, \mu)$ be
Lebesgue measure space, and let $(A_k)_{k=1}^{\infty}$ be the
decreasing sequence of the intervals $A_{k}:=[0,1/2^{k-1}]$ for
each $k\in \mathbb{N}$. Consider the \textit{integral evaluation
operator} $T\colon L^1[0,1] \to \ell^\infty$ given by
\[
T(x):= \Big(\int_{A_k} x\,d\mu \Big)_k, \quad\, x \in L^1[0,1].
\]
We claim that $T$ satisfies the assumptions of Corollary
\ref{otrafor} with $q=2=q'$, $p_1=p_2=1$ and $r_1=r_2 =2$. To see
this fix a~finite set $\{x_1,...,x_n\}$ in $L^1[0,1]$ and define
the following constants,
\[
\alpha_{0,i} := \frac{ \int_{[0,1]} |x_i|\,d\mu}{\big(
\sum_{i=1}^n \big( \int_{[0,1]} |x_i|\,d\mu \big)^2 \big)^{1/2}},
\quad 1\leq i \leq n.
\]
Note that $\big( \sum_{i=1}^n \alpha_{0,i}^2 \big)^{1/2} =1$. Then
\begin{align*}
& \inf_{(\alpha_i) \in B_{2}^n} \Big\| \sup_{1 \leq i \leq n}
\Big|\frac{Tx_i}{\alpha_i}\Big| \Big\|_{\ell^\infty}  \le \Big\|
\sup_{1 \leq i \leq n} \Big|\frac{(\int_{A_k} x_i\,d\mu)_k
}{\alpha_{i,0}} \Big| \Big\|_{\ell^\infty} \\
& \le \Big\| \sup_{1\leq i \leq n} \Big|\Big(\frac{ \int_{A_k}
x_i\,d \mu}{\int_{[0,1]} |x_i|\,d\mu} \Big)_k\Big| \,\Big(
\sum_{i=1}^n \Big(\int_{[0,1]} |x_i|\,d \mu \Big)^2 \Big)^{1/2}
\Big\|_{\ell^\infty} \\
& \le \Big( \sum_{i=1}^n\Big( \int_{[0,1]} |x_i|\,d\mu \Big)^2
\Big)^{1/2} \le \sup_{(\beta_i) \in B_{2}^n} \sum_{i=1}^n
|\beta_i | \int_{[0,1]} |x_i|\,d\mu \\
& = \sup_{(\beta_i) \in B_{2}^n} \big\| \sum_{i=1}^n |\beta_i x_i|
\big\|_{L^1[0,1]}.
\end{align*}
Thus, Corollary \ref{otrafor} applies and so $T$ is $(1,1)$-strongly
$(2,2)$-concave.

\vspace{2 mm}

\begin{theorem} \label{factolin}
Let $1 \le p_1, p_2, q_1, q_2$ be real numbers  such that $p_1 \le
q_1$ and $p_2 \le q_2=q'_1$. Let $(\Omega,\Sigma,\mu)$ be a finite
measure space. Let $X$ be an order continuous $p_1$-convex Banach
function space and $Y$  a~$p_{2}'$-concave order continuous Banach
function lattice with the Fatou property, where $1/p_2 +
1/p'_2=1$. Assume that $Y'$ is also order continuous. The following
statements about an operator $T\colon X\to Y$ are equivalent.
\begin{itemize}
\item[(i)] $T$ is $(p_1,p_2)$-strongly $(q_1,q_2)$-concave.
\item[(ii)] There is a~constant $C>0$ such that for every choice
of $(x_k)_{k=1}^n$ in $X$ and $(y_{k}^{*})_{k=1}^n$ in $Y^{*}$,
\[
\sup \bigg\{\Big|\sum_{k=1}^n \langle T(x_k), y_{k}^{*}
\rangle\Big| :\, \, \sup_{(\alpha_k) \in B_{r_2}^n}
\Big\Vert\Big(\sum_{k=1}^n |\alpha_k y_{k}^{*}|^{p_2}
\Big)^{1/{p_2}}\,\Big\Vert_{Y^*} \le 1 \bigg\}
\]
\[
\le C \sup_{(\beta_k) \in B_{r_1}^n} \Big\Vert\Big(\sum_{k=1}^n
|\beta_k x_k|^{p_1} \Big)^{1/{p_1}}\,\Big\Vert_{X}.
\]
\item[(iii)] $T$ admits the factorization
\[
\xymatrix{{X} \ar[r]^{ T} \ar@{->}[d]_{\iota_X}  & Y \,\,\,. \\
S^{q_1}_{X_{p_1}}(\nu_1) \ar[r]_{ \,\,\, T_0}  \,\,\ & (
S^{q_2}_{Y'_{p_2}}(\nu_2))'  \ar[u]_{(\iota_Y')'}  & }
\]
\end{itemize}
\end{theorem}

\vspace{2 mm}

\begin{proof}
The equivalence between (i) and (ii) is just given by Corollary
\ref{otrafor}. Let us show the equivalence of (i) and (iii).
Applying Corollary \ref{prifa}, we conclude that (i) implies that
the~bounded bilinear form on $X\times Y'$ given by
\[
(x, y') \mapsto \langle T(x), y'\rangle, \quad\, (x, y') \in X
\times Y',
\]
admits a bilinear continuous extension from the product
$S^{q_1}_{X_{p_1}}(\nu_1) \times S^{q'_2}_{(Y')_{p'_2}}(\nu_2)$ of
Banach function lattices for some probability Borel measure
spaces, i.e., there exists a~continuous bilinear form $S\colon E
\times F \to \mathbb{R}$ such that
\[
S(i_X(x),i_{Y'}(y')) = \langle T(x), y' \rangle, \quad\, (x, y') \in
X\times Y',
\]
where $E:= S^{q_1}_{X_{p_1}}(\nu_1)$, $F:= S^{q'_2}_{(Y')_{p'_2}}(\nu_2)$ and
\[
\iota_X \colon X \to E, \quad\, \iota_{Y'} \colon Y' \to F
\]
are continuous inclusions.

The required factorization follows then by using standard
arguments. At first we observe that for any fixed $f\in E$ the
formula $\langle T_0(f), \cdot \rangle := S(f,\cdot)$ defines
a~continuous functional on $F$ with
\[
\sup_{g \in F} |\langle T_0(f), g\rangle| = \sup_{g \in F} |S(f,g)|
\le \|S\| \|f\|_E\,.
\]
This clearly implies that  $T_{0}\colon E \to F$ is a~bounded linear
operator with $\|T_0\| \le \|S\|$.

Since $Y$ is $p'_2$-concave, $Y'$ is $p_2$-convex. Our
assumption on $Y'$ yields that of $(Y')_{p_2}$ is also order
continuous. Consequently, we have that the K\"othe adjoint of the
inclusion $\iota_{Y'}\colon Y' \to F$ appearing in the
factorization given by Corollary \ref{prifa} for the bilinear map
can be considered,
\[
(\iota_{Y'})' \colon F'  \to (Y')'.
\]
Combining the K\"othe duality with $Y''=Y$ (by the Fatou property)
yields the required factorization shown in (iii). The converse is
obvious.
\end{proof}

\end{document}